\documentclass[11pt]{amsart}
\usepackage{amsmath,amsxtra,amssymb,amsthm,amsfonts,showkeys}
\usepackage{color}

\numberwithin{equation}{section}
\newtheorem{lemma}{Lemma}[section]

\newtheorem{theorem}[lemma]{Theorem}

\newtheorem{cor}[lemma]{Corollary}

\newtheorem{rem}[lemma]{Remark}

\newcommand{\re}{\begin{rem}\rm}
  \newcommand{\mar}{\end{rem}}

\renewcommand{\for}{\begin{eqnarray*}}

\newcommand{\mel}{\end{eqnarray*}}

\newcommand{\nz}{{\mathbb N}}

\newcommand{\zz}{{\mathbb Z}}

\newcommand{\cz}{{\mathbb C}}

\newcommand{\ten}{\otimes}

\newcommand{\qd}{\end{proof}\vspace{0.5ex}}

\newcommand{\om}{\omega}

\newcommand{\al}{\alpha}

\newcommand{\si}{\sigma}

\newcommand{\Si}{\Sigma}

\newcommand{\eps}{\varepsilon}

\newcommand{\F}{{\mathcal F}}

\newcommand{\U}{{\mathcal U}}

\newcommand{\pf}{\begin{proof}}

\newcommand{\xspace}{\hbox{\kern-2.5pt}}
\newcommand{\xyspace}{\hbox{\kern-1.1pt}}

\newcommand{\lan}{\langle}
\newcommand{\ran}{\rangle}

\allowdisplaybreaks

\definecolor{LightGray}{rgb}{0.94,0.94,0.94}
\definecolor{VeryLightBlue}{rgb}{0.9,0.9,1}
\definecolor{LightBlue}{rgb}{0.8,0.8,1}
\definecolor{DarkBlue}{rgb}{0,0,0.6}
\definecolor{LightGreen}{rgb}{0.88,1,0.88}
\definecolor{MidGreen}{rgb}{0.6,1,0.6}
\definecolor{DarkGreen}{rgb}{0,0.6,0}
\definecolor{DarkGrreen}{rgb}{0,0.8,0}

\definecolor{VeryLightYellow}{rgb}{1,1,0.9}
\definecolor{LightYellow}{rgb}{1,1,0.6}
\definecolor{MidYellow}{rgb}{1,1,0.5}
\definecolor{DarkYellow}{rgb}{0.8,1,0.3}
\definecolor{VeryLightRed}{rgb}{1,0.9,0.9}
\definecolor{LightRed}{rgb}{1,0.8,0.8}
\definecolor{DarkRed}{rgb}{0.8,0.2,0}
\definecolor{DarkRedb}{rgb}{0.6,0.2,0}
\definecolor{DarkLila}{rgb}{0.8,0,1}
\definecolor{Beige}{rgb}{0.96,0.96,0.86}
\definecolor{Gold}{rgb}{1.,0.84,0.}
\definecolor{Goldb}{rgb}{0.7,0.3,0.5}
\definecolor{MyYellow}{rgb}{1.,0.84,0.8}

\begin{document}

\title[]{Generalized $q$-gaussian von Neumann algebras with coefficients, II. Absence of central sequences.}

\author[Marius Junge]{Marius Junge}
\address{Department of Mathematics\\
University of Illinois, Urbana, IL 61801, USA} 
\email[Marius Junge]{junge@math.uiuc.edu}

\author[Bogdan Udrea]{Bogdan Udrea}
\address{Department of Mathematics\\
University of Iowa, Iowa City, IA 52242, USA} 
\email[Bogdan Udrea]{bogdanteodor-udrea@uiowa.edu}

\begin{abstract} We show that the generalized $q$-gaussian von Neumann algebras with coefficients 
$\Gamma_q(B,S\ten H)$ with $B$ a finite dimensional factor, dim$(D_k(S))$ sub-exponential and the dimension of $H$ finite and larger than a constant depending on $q$, have no non-trivial central sequences.
\end{abstract}

\maketitle
\section{Introduction.}In this short note, which is a sequel to \cite{JungeUdreaGQC}, we investigate the lack of non-trivial central sequences in the generalized $q$-gaussian von Neumann algebras with coefficients introduced in \cite{JungeUdreaGQC}. Specifically, we prove that the von Neumann algebras $M=\Gamma_q(B,S\ten H)$ are factors without the property $\Gamma$ of Murray and von Neumann when $B$ is a finite dimensional factor, the dimensions (over $\cz$) of the spaces $D_k(S)$ (see Def. 3.18 in \cite{JungeUdreaGQC}) are sub-exponential and the dimension of $H$ is finite and larger than a constant depending on $q$. A type $II_1$ factor $(M,\tau)$ has property $\Gamma$, according to Murray and von Neumann, if there exists a sequence $(u_n)$ of unitaries in $M$ such that $\|xu_n-u_nx\|_2 \to 0$ for all $x \in M$ and $\tau(u_n)=0$ for all $n$ (see \cite{MvN}). Murray and von Neumann used this property to distinguish between the hyperfinite factor $R$ and $L(\mathbb{F}_2)$. Central sequences in type $II_1$ factors were further studied by Dixmier (\cite{DixmierQPS}) and Lance (\cite{DixmierLanceDNF}). In the 70's, property $\Gamma$ played an important role in the work of McDuff (\cite{McDuff}) and Connes (\cite{Connes76}) regarding the classification of injective factors. The absence of central sequences in the context of $q$-gaussian von Neumann algebras was investigated by Sniady (\cite{SniadyII}, see also \cite{Krolak1, Krolak2, Ricard} for the factoriality of these algebras).
\par In \cite{JungeUdreaGQC} we introduced a new class of von Neumann algebras, the so-called generalized $q$-gaussian von Neumann algebras with coefficients $\Gamma_q(B,S\ten H)$ associated to a sequence of symmetric independent copies $(\pi_j,B,A,D)$, and we proved that under certain assumptions they display a powerful structural property, namely strong solidity relative to $B$. We continue our investigation of the generalized $q$-gaussians by proving that, under the same assumptions, they do not possess the property $\Gamma$ if $B$ is finite dimensional and the dimension of $H$ is finite and exceeds a constant depending on $q$.

\section{The main theorem.}
Throughout this section we use the notations and results from Section 3 of \cite{JungeUdreaGQC}.
\begin{theorem} Let $(\pi_j,B,A,D)$ a sequence of symmetric independent copies with $B$ amenable, $1 \in S=S^* \subset A$ and assume that there exist constants $C,d>0$ such that $dim_B(D_k(S))\leq Cd^k$, for all $k \geq 0$ . Let $H$ be a Hilbert space with $2 \leq dim(H) < \infty$ and $M=\Gamma_q(B,S\ten H)$. Assume that $M$ is a factor. For $k \geq 0$, denote by $P_{\leq k}$ the orthogonal projection of $L^2(M)$ onto $\bigoplus_{s \leq k} L^2_s(M)$. Let $(x_n) \in M' \cap M^{\om}$. Then for every $\delta >0$, there exists a $k \geq 0$ such that
\[\lim_{n \to \om} \|x_n-P_{\leq k}(x_n)\|_2 \leq \delta.\]
If moreover $B$ is finite dimensional, then $M'\cap M^{\om}=\cz$, i.e. $M$ does not have the property $\Gamma$.
\end{theorem}
\begin{proof} We use the spectral gap principle of Popa (see \cite{PopaOPF, PopaSMS}). Let $\tilde M=\Gamma_q(B,S\ten(H \oplus H))$ and for every $m \geq 1$ let $\F_m \subset L^2(\tilde M)$ be the $M-M$ bimodule introduced in \cite{JungeUdreaGQC}, Sections 6 and 7. Namely, $\F_m$ is the closed linear span of reduced Wick words $W_{\si}(x_1,\ldots,x_s,h_1,\ldots,h_t) \in \tilde M$ such that $h_i \in H\oplus \{0\} \cup \{0\} \oplus H$ and at least $m$ of them are in $\{0\}\oplus H$. Also let $(\al_t)$ be the 1-parameter group of *-automorphisms of $\tilde M$ introduced in \cite{JungeUdreaGQC}, Thm. 3.16. Let's note the following transversality property, due to Avsec (see \cite{Avsec}, Prop. 5.1). 
\begin{lemma} There exists a constant $C_m>0$ such that for $0<t<2^{-m-1}$ we have
\[\|\al_{t^{m+1}}(\xi)-\xi\|_2 \leq C_m \|P_{\F_m}\al_t (\xi)\|_2 \quad \mbox{for all} \quad \xi \in \bigoplus_{k \geq m+1}L^2_k(M)\subset L^2(\tilde M).\]
\end{lemma}
As noted in Section 6 of \cite{JungeUdreaGQC}, since $B$ is amenable, there exists an $m \geq 1$ such that $\F_m$ is weakly contained into the coarse bimodule $L^2(M) \ten L^2(M)$. Fix such an $m$. Since $M$ is a non-amenable factor, it follows that $L^2(M)$ is not weakly contained in $\F_m$. This means that for every $\delta>0$ there exist a finite set $F \subset \U(M)$ and an $\eps>0$ such that if $\xi \in \F_m$ satisfies $\|u\xi-\xi u\|_2 \leq \eps$ for all $u \in F$, then $\|\xi\|_2 \leq \delta$. Fix such a $\delta$, set $\delta'=\frac{\delta}{2C_m+1}$ and take $\eps$ and $F$ corresponding to $\delta'$. Take $(x_n) \in M' \cap M^{\om}$. There's no loss of generality in assuming that $\|x_n\|_{\infty} \leq 1$ for all $n$. Fix $0<t<2^{-m-1}$ such that $\|\al_t(u)-u\|_2\leq \frac{\eps}{4}$ for all $u \in F$ and $\|\al_t(\xi)-\xi\|_2 \leq \delta'$ for all $\xi \in \bigoplus_{k \leq m}L^2_k(M)$ with $\|\xi\|_2 \leq 1$. For all $n \in \nz$ we have
\begin{align*}
& \|u\al_t(x_n)-\al_t(x_n)u\|_2=\|[\al_t(x_n),u]\|_2 = \|[x_n,\al_{-t}(u)]\|_2 \leq \|[x_n,\al_{-t}(u)-u]\|_2 + \|[x_n,u]\|_2 \\
& \leq 2\|\al_{-t}(u)-u\|_2 + \|[x_n,u]\|_2 \leq \frac{\eps}{2} + \|[x_n,u]\|_2.
\end{align*} 
Since $(x_n)\in M'\cap M^{\om}$ we see that for $n$ large enough and for all $u \in F$ we have
\[\|\al_t(x_n)u-u\al_t(x_n)\|_2 \leq \eps,\]
which further implies $\|P_{\F_m}\al_t(x_n)\|_2 \leq \delta'$. Write $x_n=x_n'+x_n''$, where $x_n' \in \bigoplus_{k \leq m}L^2_k(M)$ and $x_n'' \in \bigoplus_{k \geq m+1}L^2_k(M)$. Note that $\|x_n'\|_2 \leq 1$, $\|x_n''\|_2 \leq 1$. Due to our choice of $t$ we see that, for $n$ large enough,
\[\|P_{\F_m}\al_t(x_n')\|_2 \leq \|P_{\F_m}(\al_t(x_n')-x_n')\|_2 + \|P_{\F_m}(x_n')\|_2 \leq  \delta'.\]
Using Avsec's transversality property, this further implies, for $n$ large enough,
\[\delta' \geq \|P_{\F_m}\al_t(x_n)\|_2 \geq \|P_{\F_m}\al_t(x_n'')\|_2-\|P_{\F_m}\al_t(x_n')\|_2 \geq \|P_{\F_m}\al_t(x_n'')\|_2-\delta', \]
hence 
\[2\delta' C_m\geq C_m \|P_{\F_m}\al_t(x_n'')\|_2 \geq \|\al_{t^{m+1}}(x_n'')-x_n''\|_2.\]
Thus, for $0< s < t, t^{m+1}$ and $n$ large enough we have
\[\|\al_s(x_n)-x_n\|_2 \leq \|\al_s(x_n')-x_n'\|_2 +\|\al_s(x_n'')-x_n''\|_2 \leq (2C_m+1)\delta'.\]
Using \cite{JungeUdreaGQC}, Thm. 3.16, we see that there exists a $k=k(s,\delta)$ such that, for $n$ large enough,
\[\|x_n-P_{\leq k}(x_n)\|_2 \leq (2C_m+1)\delta'=\delta.\]
Taking the limit with respect to $n \to \om$ establishes the first statement. For the moreover part, assume first that $B=\cz$. Let's make the following general remark. Suppose $(M,\tau)$ is a type $II_1$ factor, $\om$ a free ultrafilter on $\nz$ and consider $M \subset M^{\om}$ embedded in the canonical way, i.e. as constant sequences.  Let $(x_n) \in M' \cap M^{\om}$ such that $\tau(x_n)=0$ for all $n$. Then for every $a\in M$ we have $\lim_{n \to \om}\tau(ax_n)=0$, i.e. $x_n \to 0$ ultraweakly as $n \to \om$. To prove this, let $E_M:M^{\om} \to M$ be the trace-preserving conditional expectation. Then $E_M((x_n))\in \cz$. Indeed, since $(x_n)$ is a central sequence, for every $a \in M$ we have
\[aE_M((x_n))=E_M(a(x_n))=E_M((x_n)a)=E_M((x_n))a.\]
Thus $E_M((x_n))$ is in the center of $M$, so there exists a scalar $\lambda$ such that $E_M((x_n))=\lambda 1$. Then $\lambda=\tau(E_M((x_n)))=\tau_{\om}((x_n))=\lim \tau(x_n)=0$. Hence $E_M((x_n))=0$ and for every $a,b \in M$ we have
\[\lim \tau(ax_nb)=\tau_{\om}((ax_nb))=\tau(E_M((ax_nb)))=\tau(aE_M((x_n))b)=0,\]
which proves the claim. Assume now that $(x_n) \in M'\cap M^{\om}$ such that $x_n \in\U(M)$ and $\tau(x_n)=0$ for all $n$. Fix $0<\eps<1$ and $k$ such that $\lim \|x_n-P_{\leq k}(x_n)\|_2 \leq \eps$, according to the first part of the proof. Since $D_s(S)$ is finitely generated over $B=\cz$  for every $s$, according to Prop. 3.20 in \cite{JungeUdreaGQC}, the space $\bigoplus_{s \leq k}L^2_s(M)$ is finite dimensional (over $\cz$). Choose an orthonormal basis $\{\xi_j\}_{1 \leq j \leq N(k)}$ of $\bigoplus_{s \leq k} L^2_s(M)$, then write $P_{\leq k}(x_n)=\sum_{j=1}^{N(k)} \lambda_j(n)\xi_j$, with $\lambda_j(n) \in \cz$. Note that $\sum_{j=1}^{N(k)}|\lambda_j(n)|^2=1$ for all $n$. For all $n$ large enough we have
\begin{align*}
& 1-\eps \leq |\lan x_n,P_{\leq k}(x_n) \ran|=|\lan x_n,\sum_{j=1}^{N(k)}\lambda_j(n)\xi_j \ran| \leq \sum_{j=1}^{N(k)} |\lambda_j(n)| |\lan x_n, \xi_j \ran|=\sum_{j=1}^{N(k)}|\lambda_j(n)||\tau(\xi_j^*x_n)| \to 0,
\end{align*}
which produces a contradiction. When $B$ is finite dimensional, the same argument applies since $\bigoplus_{s \leq k}L^2_s(M)$ is again finitely generated over $\cz$, and this finishes the proof.
\end{proof}
\begin{rem} The moreover statement in Thm. 2.1 can also be obtained as a consequence of \cite{OzawaSolid} and Cor. 7.5 in \cite{JungeUdreaGQC}. Indeed, due to Cor. 7.5 in \cite{JungeUdreaGQC}, the von Neumann algebras $M=\Gamma_q(B,S\ten H)$ are strongly solid under the assumptions of Thm. 2.1, hence they are also solid. Ozawa remarked in \cite{OzawaSolid}, based on a result of Popa, that a non-amenable solid factor is automatically non-$\Gamma$, which reproves the second statement of Thm. 2.1.
\end{rem}

\begin{cor} Let $-1<q<1$ be fixed. There exists $d=d(q)$ such that the following von Neumann algebras are non-$\Gamma$ factors as soon as $\infty>dim(H)\geq d$:
\begin{enumerate}
\item $\Gamma_q(H)$;
\item $B \bar{\ten} \Gamma_q(H)$, for $B$ a type $II_1$ non-$\Gamma$ factor;
\item $\Gamma_q(\cz,S \ten K)$ associated to the symmetric copies $(\pi_j,B=\cz,A=\Gamma_{q_0}(H),D=\Gamma_q(\ell^2 \ten H))$, where $-1<q_0<1$, the symmetric copies are given by $\pi_j(s_{q_0}(h))=s_q(e_j \ten h)$ ($(e_j)$ an orthonormal basis of $\ell^2$) and $K$ is a finite dimensional Hilbert space (see Example 4.4.1 in \cite{JungeUdreaGQC});
\item $\Gamma_q(\cz,S \ten H)$ associated to the symmetric copies $(\pi_j,B_0=\cz,A_0=L(\Si_{[0,1]}),D_0=L(\Si_{[0,\infty)}))$ and $S=\{1,u_{(01)}\}$; the symmetric copies are defined by $\pi_j(a)=u_{(1j)}au_{(1j)}, a \in A_0$, where $u_{\si}, \si \in \Si_{[0,\infty)}$ are the canonical generating unitaries for $D_0$ (see Example 4.4.2 in \cite{JungeUdreaGQC});
\item $\Gamma_q(\cz, S \ten H)$ associated to the symmetric copies $(\pi_j,\cz,A,D)$, where $D=\overline{\bigotimes}_{\nz}L(\zz)$ or $D=\ast_{\nz}L(\zz)$, the $j$-th copy of $L(\zz)$ is generated by the Haar unitary $u_j$, $A=\{u_1\}''$, the copies $\pi_j$ are defined by $\pi_j(u_1)=u_j$ and $S=\{1,u_1,u_1^*\}$ (see Example 4.4.3 in \cite{JungeUdreaGQC}).
\end{enumerate}
\end{cor}
\begin{proof} The von Neumann algebras in (3), (4) and (5) are factors due to Prop. 3.23 in \cite{JungeUdreaGQC}. The second statement is a consequence of Cor. 2.3 in \cite{Connes76}, while the rest follow from Thm. 2.1. Let's remark that (1) has been first proved by Sniady (\cite{SniadyII}), and that for the examples in (1) and (2) the restriction on the dimension is not necessary, due to the fact that $\Gamma_q(H)$ is a factor for dim$(H)\geq 2$ (see \cite{Ricard}) and to Remark 2.3 above.
\end{proof}

\bibliographystyle{amsplain}

\bibliographystyle{amsplain}
\bibliography{thebibliography}
\end{document}